\documentclass[a4j,11pt, leqno]{amsart}
\usepackage{amsmath,amssymb,amscd, amsthm,mathtools}
\usepackage{epic,eepic,epsfig}
\usepackage{comment} 
\usepackage{mathrsfs}
\usepackage[all,cmtip]{xy}
\usepackage{graphicx}
\usepackage{here}
\usepackage{mathtools}
\usepackage[useregional=numeric]{datetime2}
\usepackage{tikz} 
\usepackage{url} 
\tikzset{
roundnode/.style={draw,shape=circle, fill, inner sep=1pt},
roundred/.style={draw=none,shape=circle, fill=red, inner sep=2pt},
roundblue/.style={draw=none,shape=circle, fill=blue, inner sep=2pt},
pointnode/.style={draw,shape=circle,inner sep=0pt},
inner/.style={circle,draw, inner sep = 0pt},
}
\makeatletter
\def\@settitle{\begin{center}
  \baselineskip14\p@\relax
  \normalfont\LARGE\bfseries
  \@title
  \ifx\@subtitle\@empty\else
     \\[1ex] 
     
     \normalsize\mdseries\@subtitle
  \fi
 \ifx\@didication\@empty\else
     \\[2ex] 
     
     \large\mdseries\it\@dedication
  \fi
  \end{center}
}
\def\subtitle#1{\gdef\@subtitle{#1}}
\def\@subtitle{}
\def\dedication#1{\gdef\@dedication{#1}}
\def\@dedication{}
\makeatother
\usepackage{wrapfig}
\makeatletter
\makeatletter
\renewcommand{\section}{\@startsection
{section}{1}{0mm}{5mm}{2mm}{\raggedright\bfseries}}
 
  \@addtoreset{equation}{section}
\makeatother

\newtheorem{theorem}{Theorem}[section] 
\newtheorem{Theorem}[theorem]{Theorem}
\newtheorem{Lemma}[theorem]{Lemma}
\newtheorem{Corollary}[theorem]{Corollary}
\newtheorem{Proposition}[theorem]{Proposition}
\theoremstyle{definition}
\newtheorem{Definition}[theorem]{Definition}

\newtheorem{Remark}[theorem]{Remark}

\begin{document}
\newcommand{\tvect}[3]{
   \ensuremath{\Bigl(\negthinspace\begin{smallmatrix}#1\\#2\\#3\end{smallmatrix}\Bigr)}}
\newcommand{\bvect}[2]{
   \ensuremath{\Bigl(\negthinspace\begin{smallmatrix}#1\\#2\end{smallmatrix}\Bigr)}}
\newcommand{\bmatx}[4]{
   \ensuremath{\Bigl(\negthinspace\begin{smallmatrix}#1&#2\\#3&#4\end{smallmatrix}\Bigr)}}
\def\btheta{{\bar\theta}}
\def\bK{{\mathbb K}}
\def\x{{\tt x}}
\def\y{{\tt y}}
\def\eps{{\varepsilon}}
\def\al{{\epsilon}}
\def\bal{{\bar\epsilon}}
\def\I{{\tt i}} 
\def\ThetaEven{{\Theta^{\mathrm even}}}
\def\ThetaOdd{{\Theta^{\mathrm odd}}}
\def\N{{\mathbb N}}
\def\C{{\mathbb C}}
\def\Z{{\mathbb Z}}
\def\R{{\mathbb R}}
\def\Q{{\mathbb Q}}
\def\bQ{{\overline{\mathbb Q}}}
\def\Gal{{\mathrm{Gal}}}
\def\GL{{\mathrm{GL}}}
\def\et{\text{\'et}}
\def\ab{\mathrm{ab}}
\def\proP{{\text{pro-}p}}
\def\padic{{p\mathchar`-\mathrm{adic}}}
\def\la{\langle}
\def\ra{\rangle}
\def\scM{\mathscr{M}}
\def\lala{\la\!\la}
\def\rara{\ra\!\ra}
\def\ttx{{\mathtt{x}}}
\def\tty{{\mathtt{y}}}
\def\ttz{{\mathtt{z}}}
\def\bkappa{\boldsymbol \kappa}
\def\scLi{{\mathscr{L}i}}
\def\sLL{{\mathsf{L}}}
\def\cY{{\mathcal{Y}}}
\def\Coker{\mathrm{Coker}}
\def\Spec{\mathrm{Spec}\,}
\def\Ker{\mathrm{Ker}}
\def\CHplus{\underset{\mathsf{CH}}{\oplus}}
\def\check{{\clubsuit}}
\def\kaitobox#1#2#3{\fbox{\rule[#1]{0pt}{#2}\hspace{#3}}\ }
\def\vru{\,\vrule\,}
\newcommand*{\longhookrightarrow}{\ensuremath{\lhook\joinrel\relbar\joinrel\rightarrow}}
\newcommand{\hooklongrightarrow}{\lhook\joinrel\longrightarrow}
\def\nyoroto{{\rightsquigarrow}}
\newcommand{\pathto}[3]{#1\overset{#2}{\dashto} #3}
\newcommand{\pathtoD}[3]{#1\overset{#2}{-\dashto} #3}
\def\dashto{{\,\!\dasharrow\!\,}}
\def\ovec#1{\overrightarrow{#1}}
\def\isom{\,{\overset \sim \to  }\,}
\def\GT{{\widehat{GT}}}
\def\bfeta{{\boldsymbol \eta}}
\def\brho{{\boldsymbol \rho}}
\def\sha{\scalebox{0.6}[0.8]{\rotatebox[origin=c]{-90}{$\exists$}}}
\def\upin{\scalebox{1.0}[1.0]{\rotatebox[origin=c]{90}{$\in$}}}
\def\downin{\scalebox{1.0}[1.0]{\rotatebox[origin=c]{-90}{$\in$}}}
\def\torusA{{\epsfxsize=0.7truecm\epsfbox{torus1.eps}}}
\def\torusB{{\epsfxsize=0.5truecm\epsfbox{torus2.eps}}}
\title{On a two-parameter family \\
of tropical Edwards curves}
\author{Hiroaki Nakamura }
\address{Hiroaki Nakamura: 
Department of Mathematics, 
Graduate School of Science, 
Osaka University, 
Toyonaka, Osaka 560-0043, Japan}
\email{nakamura@math.sci.osaka-u.ac.jp}
\author{Rani Sasmita Tarmidi}
\address{Rani Sasmita Tarmidi:
Department of Mathematics, 
Graduate School of Science, 
Osaka University, 
Toyonaka, Osaka 560-0043, Japan}
\email{u644627d@ecs.osaka-u.ac.jp,
ranitarmidi@yahoo.com
}
\begin{abstract}
In this paper, a certain two-parameter family of plane-embeddings
of Edwards elliptic curve $E_a: x^2+y^2=a^2(1+x^2y^2)$ is introduced
to provide explicitly computed tropical curves corresponding to 
degeneration in $a\to 1$.
Applying the theta uniformization of $E_a$ with the method of
ultradiscretization by Kajiwara-Kaneko-Nobe-Tsuda, 
we give a formula for the coordinate functions
that traces the cycle part of the tropical elliptic curve.
We also illustrate how one can recover the whole part of the tropical curve 
as a quotient of the Bruhat-Tits tree after Speyer's algebraic approach
in smooth cases.
\end{abstract}
 
\maketitle
\markboth{H.Nakamura, R.S.Tarmidi}
{On a two-parameter family of tropical Edwards curves}
\tableofcontents
\footnote[0]{
This paper is originally published in 
{\it Kyushu Journal of Mathematics} {\bf 78} (2024), 373--393.
\par
\url{https://doi.org/10.2206/kyushujm.78.373}
$\copyright$ 2024 Faculty of Mathematics, Kyushu University
}

\section{Introduction}
Tropical elliptic curves on the plane have called
attentions of many authors from various viewpoints.
They generally have a unique cycle whose length is
the negative of the tropicalization of the $j$-invariant
(cf. e.g., \cite{V09}, \cite{KMM}).
In \cite{CS}, Chan and Sturmfels studied symmetric cubics
in two variables having honeycomb form tropicalizations,
whereas Nobe (\cite{N08}) closely observed a one-parameter family of tropical elliptic curves with cycles ranging over
various polygons.
In particular, 
Kajiwara-Kaneko-Nobe-Tsuda \cite{KKNT} found a beautiful
bridge from the theta functions of level 3 to 
the Hessian elliptic curves 
$E_\mu: x^3+y^3+1=3\mu xy$
which enables one to uniformize
the cycle part of the corresponding tropical curve
explicitly by 
what are called the ultradiscrete theta functions.
The purpose of this paper is to provide a simple variant of 
\cite{KKNT} in the case of Edwards curves
$E_a:x^2+y^2=a^2(1+x^2y^2)$ where only 
classical Jacobian (viz.\,level 2) theta functions
are enough to play the role for uniformization.
Our treatment mostly follows the lines of arguments in \cite{KKNT},
while, since the direct tropicalization of $E_a$ is 
never faithful on the cycle, 
we introduce a variation of the plane-embedding of $E_a$
with certain two parameters.
Our family then turns out to contain fairly rich  
plane elliptic curves (isomorphic to $E_a$)
whose tropical cycles range over $n$-gons
($n=4,5,7$) with explicit uniformization by 
ultradiscrete theta functions.
We now illustrate our main results.
Let 
$$
\epsilon=\epsilon(q)=
\prod_{n=1}^\infty (1+q^n)
\ 
\left(=\prod_{n=1}^\infty\frac{1}{1-q^{2n-1}}=1+q+q^2+2q^3+\cdots
\right)
$$
be the Euler generating function counting the
number of partitions of $n$
with distinct parts (which is the same as
the number of partitions of $n$ with odd parts;
see \cite[(1.2.5)]{A84}), and set 
$$
\bal=
\bal(q):=\epsilon(-q) \ 
\left(=\prod_{n=1}^\infty\frac{1}{1+q^{2n-1}}=1-q+q^2-2q^3+\cdots
\right).
$$
Let $K$ be a complete discrete valuation field of characteristic 0
with normalized
valuation $v_K:K\twoheadrightarrow \Z\cup\{\infty\}$.
Pick and fix an element $q\in K$ with $v_K(q)>0$ and
consider $\al$, $\bal\in K$ to be the convergent limits 
of the above generating functions respectively.
Let $\bK$ be the algebraic closure of $K$ which has a 
unique valuation $v_\bK:\bK\twoheadrightarrow \Q\cup\{\infty\}$
extending $\frac{1}{v_K(q)} v_K$ so that $v_\bK(q)=1$.
For two parameters $r,s\in\bK$ with $\al r\ne \bal s$,
let us consider a polynomial
\begin{equation}\label{f_rs}
f_{r,s}(\x,\y)=
d_{12}(\x+\y)+d_{34}(\x^2+\y^2)+d_5 \x \y
+d_{67}(\x^2 \y+\y^2 \x)+d_8 \x^2\y^2 
\end{equation}
in two variables $\x,\y$
with
\begin{equation}
\label{d_ij}
\begin{cases}
d_{12}&=
2\al\bal(\al^4-\bal^4)(\bal s-\al r),
 \\
d_{34}&= 
(\al^4-\bal^4)(\bal^2s^2-\al^2r^2),
\\
d_5&= 
8\al\bal(\al r-\bal s)(\bal^3 r-\al^3 s),
\\
d_{67} &= 
2(\al r-\bal s)
\{ (\bal^4-\al^4)rs +2\al\bal (\bal^2r^2-\al^2s^2)\},
\\
d_8 &=
2 (\al^2 s^2-\bal^2 r^2)(\bal^2 s^2- \al^2 r^2).
\end{cases}
\end{equation}
We have then:
\begin{Proposition}
\label{thm1.1}
The equation
$f_{r,s}(\x,\y)=0$ defines an elliptic curve over $\bK$
birationally equivalent
to the Edwards curve 
$E_a:x^2+y^2=a^2(1+x^2y^2)$
with
$\displaystyle a^2=\frac{2\al^2\bal^2}{\al^4+\bal^4}\in K$. 
The j-invariant is given by:
$$
j(q^8)=\frac{1}{q^8}+744+196884q^8+\cdots 
$$
with $j(q)$ the standard $q$-series for
the $j$-invariant.
\end{Proposition}
By virtue of the known relation between 
$j$-invariant and tropical cycle length
(cf.\,\cite{V09}, \cite{KMM}),
the above Proposition \ref{thm1.1} implies 
that the tropicalization
of the plane curve $f_{r,s}(\x,\y)=0$
has a unique cycle of length 8, if it is tropically smooth.
Write $u_{12},u_{34},u_5,u_{67},u_8\in\Q\cup\{\infty\}$ for the values
$v_\bK(d_{12}), v_\bK(d_{34}), v_\bK(d_5), v_\bK(d_{67}), v_\bK(d_8)$ respectively.
Our primary concern is the tropical curve $C(\mathrm{trop}(f_{r,s}))$ 
on the $XY$-plane
which is 
by definition the graph  obtained as the set of points
$(X,Y)\in\R^2$ where the piecewise linear function 
$\mathrm{trop}(f_{r,s}):\R^2\to\R$ with
\begin{equation*}
(X,Y)\mapsto
\mathrm{trop}(f_{r,s})(X,Y):=\mathrm{min}
\left\{
\begin{matrix}
u_{12}+X,u_{12}+Y,u_{34}+2X,u_{34}+2Y,u_5+X+Y, \\
\quad u_{67}+2X+Y,u_{67}+2Y+X,u_8+2X+2Y
\end{matrix}
\right\}
\end{equation*}
is not differentiable.
Our first result on the family $C(\mathrm{trop}(f_{r,s}))$ concerns  
an explicit parametrization of its cycle part
(viz. the maximal subgraph with no end points)
in terms of a pair of  `ultradiscrete'
theta functions
$\ThetaOdd(u)$, $\ThetaEven(u)$ defined in the same spirit as
\cite{KKNT}:
\begin{Theorem}
\label{thm1.2}
Let $C(\mathrm{trop}(f_{r,s}))$ be
the tropicalization of the plane curve $f_{r,s}(\x,\y)=0$
for  $r,s\in \bK$ $(\al r\ne \bal s)$.
Then, the cycle part of $C(\mathrm{trop}(f_{r,s}))$ has one-parameter expression 
$(-X(u),-Y(u))_{u\in\R}$ as follows:
$$
\begin{cases}
X(u)=& Y(u-\frac12),  \\
Y(u)=& \mathrm{max}\left(\ThetaOdd(u),-1+\ThetaEven(u)\right) \\
&\ 
-\mathrm{max}\left(-v_\bK(r-s)+\ThetaEven(u),-v_\bK(r+s)+\ThetaOdd(u)
\right),
\\
\end{cases}
$$
where 
$\ThetaOdd(u):=-2(2\lfloor \frac{u}{2}\rfloor +1 -u)^2$,
$\ThetaEven(u):=-2(2\lfloor \frac{u+1}{2}\rfloor -u)^2$. 
\end{Theorem}
As an immediate application of the above theorem, it follows that
the shape of the cycle part of  $C(\mathrm{trop}(f_{r,s}))$ relies 
only on the value of $v_\bK(r+s)-v_\bK(r-s)$. More concretely:
\begin{Corollary}
\label{cor1.3}
Under the same notations and assumptions as in Theorem \ref{thm1.2},
set $$
\delta(=\delta_{r,s}):=v_\bK(r+s)-v_\bK(r-s).
$$
Then, 
we have the following assertions.
\begin{enumerate}
\item[(i)] The curve $C(\mathrm{trop}(f_{r,s}))$ 
has a pentagonal cycle of length $8$ if and only if
$2\le \delta$.
\item[(ii)] The curve $C(\mathrm{trop}(f_{r,s}))$ 
has a heptagonal cycle of length $8$ if and only if
$1< \delta <2$.
\item[(iii)] The curve $C(\mathrm{trop}(f_{r,s}))$ 
has a square cycle of length 
$4(\delta+1)$ if and only if
$-1<\delta \le 1$.
In particular, it has a square cycle of length 8 
if and only if 
$\delta = 1$.
\item[(iv)]
If $\delta \le -1$, then the locus of 
$(-X(u),-Y(u))_{u\in\R}$
degenerates to a connected union of two segments 
of length $\mathrm{min}(1, -\delta -1)$.
\end{enumerate}
\end{Corollary}
Our next result concerns a criterion when 
$C(\mathrm{trop}(f_{r,s}))$ is tropically smooth. 
We obtain:
\begin{Proposition}
\label{prop1.4}
Notations being as in Corollary \ref{cor1.3}, 
the following assertions hold.
\begin{enumerate}
\item[(i)] 
If $2\le \delta$, then $C(\mathrm{trop}(f_{r,s}))$ 
is never a smooth tropical curve.
\item[(ii)] 
If $1< \delta <2$, then 
$C(\mathrm{trop}(f_{r,s}))$ is always a smooth
tropical curve.
\item[(iii)] If $\delta=1$, then 
the curve $C(\mathrm{trop}(f_{r,s}))$ 
can be either a smooth curve or a non-smooth curve
according to the choice of $(r,s)$. 
It is smooth if and only if the principal coefficient
of $r+s$ equals $-2$, i.e., $r+s$ is 
of the form $q^{v_\bK(r+s)}(-2+\bkappa)$ 
for some $\bkappa\in\bK$ with 
$v_\bK(\bkappa)>0$.
\item[(iv)]
If $\delta <1$, then $C(\mathrm{trop}(f_{r,s}))$ 
is never a smooth tropical curve.
\end{enumerate}
\end{Proposition}

The contents of this paper are organized as follows. 
In Sect.2, we review the basic setup on the family of 
Edwards elliptic curves $E_a: x^2+y^2=a^2(1+x^2y^2)$
and their parametrization 
in terms of the classical Jacobian theta functions.
We then introduce our above
two-parameter family (\ref{f_rs}) of elliptic curves $f_{r,s}(\x,\y)=0$
as fractional transformations of $E_a$
at a specific localization $a\to 1$ given in Proposition \ref{thm1.1}.
In Sect.3, we compute its
ultradiscretization by following the method
of \cite{KKNT}. This yields the essential part of the formula in
Theorem 1.2. Then, in Sect.4, we verify Theorem 1.2
in the general setting of the base field $\bK$ and deduce Corollary 1.3.
In Sect.5, we turn to investigate the smoothness criterion
for $C(\mathrm{trop}(f_{r,s}))$ by looking closely at 
subdivisions of Newton polytopes, and prove Proposition \ref{prop1.4}.
Sect.6 is devoted to illustrating examples of smooth tropical
curves $C(\mathrm{trop}(f_{r,s}))$ to be given as the quotient
of Speyer's subtree of the Bruhat-Tits tree corresponding 
to Tate uniformization $\bK^\times/\la q^8\ra\cong E_{r,s}(\bK)$.

Throughout this paper, we shall write $\I:=\sqrt{-1}\in\C$.

\bigskip
{\it Acknowledgement}:
The authors would like to thank the referee for careful reading and valuable comments. The second named author is grateful 
to The Indonesia Endowment Funds 
for Education (LPDP) for generous supports.
This work was partially supported by JSPS KAKENHI Grant Number 
JP20H00115.

\section{Algebraic and analytic theory of Edwards curves}
In \cite{E07}, H.Edwards introduced the normal form of elliptic curves
$E_a:x^2+y^2=a^2(1+x^2y^2)$
and established the basic theory. Every complex elliptic curve 
is isomorphic to $E_a$ 
for some $a\in\C-\{0,\pm 1,\pm \I\}$ which has
a simple symmetric addition law with regard to the
origin $O=(0,a)$.
The space $\C-\{0,\pm 1,\pm \I\}$ of the parameter $a$ is
in fact the set of complex points of the modular curve
$Y(4)$ which can be identified with a dense open subset of
the degree-$2$ Fermat curve $X^2+Y^2=1$
with $(X,Y)=(\frac{2a}{1+a^2},\frac{1-a^2}{1+a^2})$
(cf.\,\cite[Chap.\,I, \S 10]{M06}).
The algebraic theory mostly works without big changes
over any field of characteristic different from 2. 
\begin{Remark}
The simplicity of addition law of Edwards curves has 
attracted cryptographic studies, e.g., 
\cite{BL07}.
See also \cite{GS14} for advantages of tropicalization
in view of efficiency of computations and of security
against various network attacks.
\end{Remark}
One of the primary features elaborated in \cite{E07}
is a complex uniformization of the curve 
$E_a: x^2+y^2=a^2(1+x^2y^2)$
by the upper half plane
$\mathfrak{H}=\{\tau\in\C\mid \mathrm{Im}(\tau)>0\}$.
We shall rephrase it in terms of standard theta functions:
{\small
\begin{equation}
\label{theta-def}
\begin{dcases}
\theta_1(z|\tau)&=-\I \sum_{n\in\Z}
(-1)^n q_\tau^{(\frac12+n)^2}q_z^{2n+1}
\left(=-\I (q_z-q_z^{-1})
\prod_{n=1}^\infty (1-q_\tau^{2n})(1-q_\tau^{2n}q_z^2)(1-q_\tau^{2n}q_z^{-2}) 
\right), \\
\theta_2(z|\tau)&= \sum_{n\in\Z}
q_\tau^{(\frac12+n)^2}q_z^{2n+1}
\left(=q_\tau^{1/4}
 (q_z+q_z^{-1})
\prod_{n=1}^\infty (1-q_\tau^{2n})(1+q_\tau^{2n}q_z^2)(1+q_\tau^{2n}q_z^{-2})
\right) , \\
\theta_3(z|\tau)&= \sum_{n\in\Z}
q_\tau^{n^2}q_z^{2n} 
\left(=\prod_{n=1}^\infty (1-q_\tau^{2n})(1+q_\tau^{2n-1}q_z^2)(1+q_\tau^{2n-1}q_z^{-2})
\right), \\
\theta_4(z|\tau)&= \sum_{n\in\Z}
(-1)^n q_\tau^{n^2}q_z^{2n} 
\left(=\prod_{n=1}^\infty (1-q_\tau^{2n})(1-q_\tau^{2n-1}q_z^2)(1-q_\tau^{2n-1}q_z^{-2})\right),
\end{dcases}
\end{equation}
}
\noindent
where $q_\tau=\exp(\pi \I \tau)$, $q_z=\exp(\pi \I z)$
for $\tau\in\mathfrak{H}$, $z\in\C$.
\begin{Proposition}[\cite{E07} Theorem 15.1]
\label{EdwardsThm}
For any fixed $\tau\in\mathfrak{H}$, set
$$
a=a(\tau):=\frac{\theta_2(0|2\tau)}{\theta_3(0|2\tau)}.
$$ 
and let 
$$
x(z):=\frac{\theta_1(z|2\tau)}{\theta_4(z|2\tau)},\qquad
y(z):=\frac{\theta_2(z|2\tau)}{\theta_3(z|2\tau)}.
$$
Then we have
$$
x(z)^2+y(z)^2=a^2(1+x(z)^2y(z)^2)
$$
for all $z\in \C$.
The mapping of the complex $z$-plane $\C_z$ to
the complex points $(x(z),y(z))$ of $E_a$
gives a uniformization of
the elliptic curve $:\C_z\twoheadrightarrow
\C_z/(2\Z+2\tau\Z)\isom E_a(\C)$. \qed
\end{Proposition}
Our motivating idea is to variate the equation of $E_a$
by fractional substitutions of the form
\begin{equation}
\label{substitute}
x=\frac{r\x+\alpha}{s\x+\beta}, \quad
y=\frac{r\y+\alpha}{s\y+\beta}
\end{equation}
for some constants $r,s,\alpha,\beta$ 
($\alpha s \ne \beta r$) to obtain
nicer equations in regards of tropicalization.
Let 
\begin{equation}
\label{general-equation}
f_{r,s}^{\alpha,\beta}(\x,\y)
=d_0+d_{12}(\x+\y)+d_{34}(\x^2+\y^2)+d_5 \x \y
+d_{67}(
{\color{black} \x^2\y+\y^2\x})+d_8 \x^2\y^2 
\end{equation}
be the numerator of the rational function
$x^2+y^2-a^2(1+x^2y^2)$ in $\x,\y$ obtained by 
the variable change from $(x,y)$ to $(\x,\y)$
after (\ref{substitute}).
\begin{Lemma} \label{lemma_d0}
The constant term $d_0$ of $f_{r,s}^{\alpha,\beta}(\x,\y)$
is given by
$$
d_0=2\alpha^2\beta^2-a^2(\alpha^4+\beta^4).
$$
\end{Lemma}
In order to obtain effective tropical curves, 
following an idea of \cite{KKNT}, we shall change the focus of ultradiscretization
process from the standard limit $\tau\to \I \infty$ along the imaginary axis 
to the limit $\tau\to 0$ along the semicircle emanating from $\frac14$.
We will implement this idea by, roughly speaking, replacing $q$-expansions 
of relevant analytic functions in $q_\tau$ 
by those in 
\begin{equation}
\label{OurQ}
q:=\exp(\pi\I \frac{4\tau-1}{4\tau}).
\end{equation}
\begin{Lemma} \label{lemma_eps}
$$
2\al(q)^2\bal(q)^2-
\left(\frac{\theta_2(0|2\tau)}{\theta_3(0|2\tau)}\right)^2
(\al(q)^4+\bal(q)^4)=0.
$$
\end{Lemma}
\begin{proof}
Applying the Landen type transformations 
(cf.\,\cite[(1.8.5-6) \& Ex.2 of Chap.1]{L10}), we find
$$
\left(\frac{\theta_2(0|2\tau)}{\theta_3(0|2\tau)}\right)^2
=
\frac{\theta_3(0|\tau)^2-\theta_4(0|\tau)^2}{\theta_3(0|\tau)^2+\theta_4(0|\tau)^2}
=
\frac{2\theta_2(0|4\tau)\theta_3(0|4\tau)}{\theta_3(0|4\tau)^2+\theta_2(0|4\tau)^2}.
$$
Combining this with the theta transformation (cf. \cite[p.482 (5)]{F1916})
with $4\tau\to \frac{4\tau-1}{4\tau}$ in the form:
$$
\frac{\theta_2(0|4\tau)}{\theta_3(0|4\tau)}
=
\frac{\theta_3(0|\frac{4\tau-1}{4\tau})}{\theta_4(0|\frac{4\tau-1}{4\tau})}
=\prod_{n=1}^\infty
\frac{(1-q^{2n})(1+q^{2n-1})^2}{(1-q^{2n})(1-q^{2n-1})^2}
=\frac{\al(q)^2}{\bal(q)^2},
$$
we obtain 
\begin{equation} \label{a(q)}
\left(\frac{\theta_2(0|2\tau)}{\theta_3(0|2\tau)}\right)^2
=\frac{2}{(\al(q)/\bal(q))^2+(\bal(q)/\al(q))^2}
=\frac{2\al(q)^2\bal(q)^2}{\al(q)^4+\bal(q)^4}
\end{equation}
for our $q$ given in (\ref{OurQ}). 
This proves the assertion. 
\end{proof}
In the sequel, we set the parameters 
$\alpha=\bal(q)$ and $\beta=\al(q)$ 
in the fractional substitution (\ref{substitute}),
so that the resulting polynomial (\ref{general-equation}) 
has no constant term $d_0$
by virtue of Lemmas \ref{lemma_d0}-\ref{lemma_eps}.
Accordingly, we restrict ourselves to focusing on the
two-parameter family
\begin{equation}
\label{special-equation}
f_{r,s}(\x,\y):=
f_{r,s}^{\bal(q),\al(q)}(\x,\y)
=d_{12}(\x+\y)+d_{34}(\x^2+\y^2)+d_5 \x \y
+d_{67}(\x^2 \y+\y^2\x)+d_8 \x^2\y^2 
\end{equation}
with $r,s$ given as Laurent power series in 
$q^{1/N}$ ($N\in \N$) convergent around (possibly with poles at) $q=0$. 
Note that $\al r\ne \bal s$ is assumed as
the non-degeneracy condition 
for the substitutions (\ref{substitute}).
The coefficients 
$d_{12},d_{34},d_5,d_{67},d_8$ are derived 
in the form (\ref{d_ij}) by a simple computation
(using, say, Maple \cite{Maple}).
\begin{proof}[Proof of Proposition \ref{thm1.1}]
By construction, the plane curve $f_{r,s}(\x,\y)$ is
birationally equivalent to the Edwards curve
$E_a: x^2+y^2=a^2(1+x^2y^2)$ with $a^2=\frac{2\al^2\bal^2}{\al^4+\bal^4}$.
In \cite{E07}, the $j$-invariant of $E_a$ is known to be
$$
j(E_a)=\frac{1728}{108}\cdot 
\frac{(a^8+14a^4+1)^3}{a^4(a^4-1)^4}.
$$
If $a=a(\tau)$ is given as in Proposition 
\ref{EdwardsThm} in terms of theta-zero values, then
it follows by simple calculation that 
$j(E_a)=j(\tau)=j(\frac{4\tau-1}{\tau})$
where the latter equality is due to the 
$\mathrm{PSL}_2(\Z)$-invariance of $j$-function.
Thus
$$
j(E_a)=j\left(\exp(2\pi \I \frac{4\tau-1}{\tau})\right)
$$
with $j(\ast)$ in RHS being the standard Fourier 
series of the $j$-function 
$j(z)=\frac{1}{z}+744+196884z+\cdots $
in $z=e^{2\pi \I \tau}$.  
Since
$\exp(2\pi \I \frac{4\tau-1}{\tau}))=q^8$ for 
$q$ given in (\ref{OurQ}), the assertion follows.
\end{proof}
\section{Ultra-discretization}
\label{sect.3}
The power series $\al(q),\bal(q)$ converge on the Poincar\'e disk
$|q|<1$. 
In this section, we assume the two parameters $r(q),s(q)$ are
contained in $\C\{q^{1/N}\}[\frac{1}{q}]$,
the fractional field of the convergent power series ring
in $q^{1/N}$
for some large integer $N\in\N$.
We shall pursue the ultradiscretization of 
the points $(x,y)$ with
$$
x=x(z)=\frac{\theta_1(z|2\tau)}{\theta_4(z|2\tau)},\qquad
y=y(z)=\frac{\theta_2(z|2\tau)}{\theta_3(z|2\tau)}
$$
on the Edwards curve 
$$
x^2+y^2=\left(
\frac{\theta_2(0|2\tau)}{\theta_3(0|2\tau)}\right)^2
(1+x^2y^2)
$$
at the limit 
\begin{equation}
\label{OurQto0}
q=\exp(\pi\I \frac{4\tau-1}{4\tau})\to 0.
\end{equation}
Our method mostly follows the argument given in \cite{KKNT}
for the Hessian elliptic curves.
Introduce an adjustment constant $\theta\in \R_{>0}$
for the above limit in the form
\begin{equation}
\label{adjust-constant}
\frac{4\tau-1}{4\tau}=\frac{\I \theta}{\eps}\to \I\infty 
\quad (\eps\to 0).
\end{equation}
Let us first observe the behaviors of 
$x(z), y(z)$ with $z=u+\I v$ ($u,v\in\R$)
under the Fourier expansion (\ref{theta-def}) in $q$.
Note first that the theta transformation
\cite[p.482 (5)]{F1916} yields
\begin{equation}
\label{FrickeTF}
x(z)=\frac{\theta_1(z|2\tau)}{\theta_4(z|2\tau)}
=-\I \frac{\theta_1(\frac{-z}{2\tau}|\frac{2\tau-1}{2\tau})}{\theta_2(\frac{-z}{2\tau}|\frac{2\tau-1}{2\tau})} , \quad
y(z)=\frac{\theta_2(z|2\tau)}{\theta_3(z|2\tau)}
=\frac{\theta_3(\frac{-z}{2\tau}|\frac{2\tau-1}{2\tau})}{\theta_4(\frac{-z}{2\tau}|\frac{2\tau-1}{2\tau})} ,
\end{equation}
and our above setting (\ref{OurQto0})-(\ref{adjust-constant})
implies  
\begin{equation}
\exp(\pi\I \frac{2\tau-1}{2\tau})=-q^2
=\exp\left[\pi\I (-1+\frac{2\theta \I}{\eps})\right]
,\quad
\frac{u+\I v}{2\tau}=2(u+\I v)(1-\frac{\I \theta}{\eps}),
\end{equation}
which enables us to evaluate the Fourier expansion of
$x(z),y(z)$.
Let us first look at the numerator of $y(z)$ as follows:
\begin{align*}
\theta_3(\frac{-z}{2\tau}|\frac{2\tau-1}{2\tau})
&=
\sum_{n\in\Z}
\exp\left[n^2\pi\I (-1+\frac{2\theta \I}{\eps})
+2\pi\I n(-u-v\I)2(1- \frac{\I \theta}{\eps})
\right] \\
&=
\sum_{n\in\Z}
\exp\left[
-\pi
\left(
\frac{2\theta}{\eps}n^2
+4n(\frac{\theta}{\eps}u-v)
+\I (n^2+4n(u+\frac{\theta}{\eps}v))
\right)
\right]
.
\end{align*}
In the ultradiscretization process, the imaginary exponents
should be annihilated. This determines the imaginary part
$v$ of $z$ to be equal to $-u \eps/\theta$:
\begin{equation}
z=u-\I \eps \frac{u}{\theta}.
\end{equation}
Since $\exp(-n^2\pi\I)=(-1)^n$, this implies
\begin{equation*}
\theta_3(\frac{-z}{2\tau}|\frac{2\tau-1}{2\tau})
\sim
\sum_{n\in\Z} (-1)^n
\exp\left[
-\frac{2\pi \theta}{\eps}
\bigl((n+u)^2-u^2)\bigr)
\right]
\end{equation*}
as $\eps\to 0$.
After similar calculations for the denominator of $y(z)$ and
the numerator and denominator of $x(z)$, we summarize 
evaluations as follows:
\begin{equation}
\label{xy-eval0}
\begin{dcases}
x(u-\I \eps \frac{u}{\theta})
&\sim{}
-
\frac{
\sum_{n\in\Z} (-1)^n
\exp\left[
-\frac{2\pi \theta}{\eps}
(n+u+\frac12)^2
\right]}
{\sum_{n\in\Z}
\exp\left[
-\frac{2\pi \theta}{\eps}
(n+u+\frac12)^2
\right]}
\\
&\qquad={}
\frac{
\sum_{n\in\Z} (-1)^n
\exp\left[
{\color{black} -}
\frac{2\pi \theta}{\eps}
(n+u-\frac12)^2
\right]}
{\sum_{n\in\Z}
\exp\left[
-\frac{2\pi \theta}{\eps}
(n+u-\frac12)^2
\right]},
\\
y(u-\I \eps \frac{u}{\theta})
&\sim{} 
\frac{\sum_{n\in\Z} (-1)^n
\exp\left[
-\frac{2\pi \theta}{\eps}
(n+u)^2
\right]}
{\sum_{n\in\Z}
\exp\left[
-\frac{2\pi \theta}{\eps}
(n+u)^2
\right]}.
\end{dcases}
\end{equation}
Here we note that the factors $\exp(\frac{2\pi \theta}{\eps}(u^2+\frac14))$ 
(resp.  $\exp(\frac{2\pi \theta}{\eps}u^2)$) 
from $\theta_1,\theta_4$ (resp. from $\theta_2,\theta_3$)
are cancelled out 
in the numerator and denominator
in the above right hand sides' expressions
(and that $\exp(-\pi \I (n+\frac12)^2)=-\I$ and
$-(-1)^n=(-1)^{n+1}$ for the former expressions from $\theta_1,\theta_2$).
We next convert the above evaluations 
(\ref{xy-eval0})
of $x,y$ to
those of $\x,\y$ under the substitutions
(\ref{substitute}) with $\alpha=\bal(q)$ and $\beta=\al(q)$ 
via
\begin{equation}
\x=\frac{\al x-\bal}{-s x +r},\quad
\y=\frac{\al y -\bal}{-s y+r}. 
\end{equation}
Write $\al(q)=\sum_{k=0}^\infty a_k q^k$ so that
$\bal(q)=\sum_{k=0}^\infty (-1)^k a_k q^k$
(where we know $a_k>0$ for all $k\ge 0$)
and suppose that the two parameters $r,s\in\C\{q^{1/N}\}[\frac{1}{q}]$ 
are
given in the form 
\begin{equation}
r=\sum_{k=0}^\infty r_k q^{a+\frac{k}{N}}=
q^a (r_0+r_1 q^{1/N}+ \cdots),
\quad
s=\sum_{k=0}^\infty s_k q^{a+\frac{k}{N}}=q^a(s_0+s_1 q^{1/N}+\cdots)
\end{equation}
with
$a\in\frac{1}{N}\Z$, $|r_0|+|s_0|\ne 0$.
Then, we obtain:
{\small
\begin{equation}
\label{xy-eval}
\begin{dcases}
\x(u-\I \eps \frac{u}{\theta}) 
&\sim\ 
\frac{
P(u-\frac12,\theta,\eps)
\left(\sum_{k:odd\ge 1}2a_k q^k\right)
-
Q(u-\frac12,\theta,\eps)
\left(\sum_{k:even\ge 0} 2a_k q^k\right)
}
{P(u-\frac12,\theta,\eps)
\left(
\sum_k(r_k-s_k)q^{a+\frac{k}{N}}\right)
+
Q(u-\frac12,\theta,\eps)
\left(\sum_k(r_k+s_k)q^{a+\frac{k}{N}}\right)
},
\\
\y(u-\I \eps \frac{u}{\theta})
&\sim\ 
\frac{
P(u,\theta,\eps)
\left(\sum_{k:odd\ge 1} 2a_k q^k\right)
-
Q(u,\theta,\eps)
\left(\sum_{k:even\ge 0} 2a_k q^k\right)
}
{P(u,\theta,\eps)
\left(\sum_k(r_k-s_k)q^{a+\frac{k}{N}}\right)
+
Q(u,\theta,\eps)
\left(\sum_k(r_k+s_k)q^{a+\frac{k}{N}}\right)
},
\end{dcases}
\end{equation}
}
\\
\noindent
where
$P(u,\theta,\eps):=\sum_{n:even}\exp[
{\color{black} -}
\frac{2\pi\theta}{\eps}(n+u)^2]$, 
$Q(u,\theta,\eps):=\sum_{n:odd}\exp[
{\color{black} -}
\frac{2\pi\theta}{\eps}(n+u)^2]$.
At this stage,
we are led to define the ultradiscrete theta functions
$\ThetaEven$ and $\ThetaOdd$ as those 
limits of the quantities 
$P(u,\theta,\eps) $ and  $Q(u,\theta,\eps)$ respectively:
\begin{Definition}
For $u\in \R$, define
\begin{align*}
\ThetaEven(u):=\lim_{\eps\to 0+} \eps \log P(u,\theta,\eps)
&=
2\pi\theta \max_{n:even}(-(n+u)^2)
=-2\pi\theta
\left[2\lfloor \frac{u+1}{2} \rfloor-u\right]^2, \\
\ThetaOdd(u):=\lim_{\eps\to 0+} \eps \log Q(u,\theta,\eps)
&=
2\pi\theta \max_{n:odd}(-(n+u)^2)
=-2\pi\theta \left[2\lfloor \frac{u}{2} \rfloor +1-u\right]^2.
\end{align*}
\end{Definition}
\noindent
Note here that the distance from any $u\in\R$ to the nearest 
even (resp. odd) integer can be written as 
$|2\lfloor \frac{u+1}{2} \rfloor-u|$
(resp. $|2\lfloor \frac{u}{2} \rfloor +1-u|$).
It is easy to see that these functions $\ThetaEven,\ThetaOdd$ 
are even functions of period 2 on $\R$ and satisfy 
$\ThetaEven(\pm u\pm 1)=\ThetaOdd(\pm u)$.
With the above definition of $\ThetaEven,\ThetaOdd$
together with  (\ref{xy-eval}),
we conclude that the ultradiscrete limit
$(X(u),Y(u))$ of the point $(\x(u-\I \eps u/\theta),
\y(u-\I \eps u/\theta))$
is given by:
\begin{equation}
\begin{dcases}
X(u)&=\lim_{\eps\to 0+} 
\eps \log
\x(u-\I \eps \frac{u}{\theta})=Y(u-\frac12),
  \\
Y(u)&= 
\lim_{\eps\to 0+} 
\eps \log
\y(u-\I \eps \frac{u}{\theta}) \\
&=\mathrm{max}\left(\ThetaOdd(u),-1+\ThetaEven(u)\right) \\
&\quad
-\mathrm{max}
\left(-{\color{black} v_\mathbb{K}}(r-s)+\ThetaEven(u),
-{\color{black} v_\mathbb{K}}(r+s)+\ThetaOdd(u)\right).
\end{dcases}
\end{equation}
\begin{Remark}
The minus sign problem for ultradiscrete limits can be avoided
in our case. See \cite[Remark 2.1]{KNT08} (cf.\,also \cite{KL06}).
\end{Remark}
Now let us normalize our adjustment constant $\theta$
introduced above in (\ref{adjust-constant})
by comparing our ultradiscrete limit with
the non-archimedean amoeba 
studied in Einsiedler-Kapranov-Lind \cite{EKL06}
and in Speyer \cite{S14}.
In fact, our fundamental
quantities $\al(q),\bal(q)$ and the two parameters $r(q),s(q)$
are considered as elements of the field of 
convergent Laurent series $\C\{q^{1/N}\}[\frac{1}{q}]$
for taking analytic limits $q\to 0$, however 
we are able to enhance this analytic procedure 
to a more general algebraic process 
via the valuation theory:
Consider $\C\{q^{1/N}\}[\frac{1}{q}]$ as a subfield of 
the standard Puiseux power series $\C\{\!\{q\}\!\}$
which has the standard valuation
$v:\C\{\!\{q\}\!\}\to \Q\cup\{\infty\}$ with $v(q)=1$
and has the non-archimedean norm $||a||:=e^{-v(a)}$
($a\in \C\{\!\{q\}\!\}$). According to \cite{EKL06}, 
the non-archimedean amoeba 
$\mathcal{A}(S)\in\R^2$ of a subscheme 
$S\subset (\C\{\!\{q\}\!\}^\times)^2$
is by definition the closure of the image of $S$ by
the map $\mathrm{Log}:(\C\{\!\{q\}\!\}^\times)^2\to \R^2$ 
($(\x,\y)\mapsto (\log ||\x||,\log||\y||)$), and 
the tropical variety $\mathrm{Trop}(S)\subset \R^2$ is the
closure of the image of $X$ by the map
$\mathrm{val}:(\C\{\!\{q\}\!\}^\times)^2\to \R^2$ defined by
$(\x,\y)\mapsto (v(\x),v(\y))$; this is equivalent
to saying that 
\begin{equation}
\mathrm{Trop}(S)=-\mathcal{A}(S).
\end{equation}
As given in (\ref{OurQto0})-(\ref{adjust-constant}), 
our ultradiscrete limits are taken with respect to 
the base scaling $q=\exp(-\pi\theta/\eps)$ $ (\eps\to 0)$, 
in particular $\eps\log q \to-\pi \theta$.
Comparing this with $v(q)=1$, $\log||q||=-1$ in the non-archimedean
metric of $\C\{\!\{q\}\!\}$, we  
are led to normalize our above adjustment
constant $\theta$ as follows:
\begin{equation}
\label{normalized-theta}
\theta :=1/\pi.
\end{equation}
This enables us to identify the plane curve
$\{(X(u),Y(u))\mid{u\in\R} \}\subset \R^2$ to lie on
the cycle part of the 
non-archimedean amoeba $\mathcal{A}(f_{r,s}(\x,\y)=0)$,
viz., 
the cycle part of the 
closure of
\begin{equation}
\mathrm{Log}\left(
\Bigl\{(||\x||,||\y||)\left| f_{r,s}(\x,\y)=0, \x\y\ne 0 ,
\ 
(\x,\y)\in \C\{\!\{q\}\!\}^2 \Bigr\}\right.\right)
\subset \Q^2.
\end{equation}
\section{Proofs of Theorem 1.2 and Corollary 1.3}
\label{algebraicProof}
Now we convert the analytic theory of Edwards curve to algebraic
theory over the complete discrete valuation field $K$ 
of characteristic 0
with the
algebraic closure $\bK$ as set up in Introduction.  
The Edwards curve
$$
E_a: x^2+y^2=a^2(1+x^2y^2)
$$ 
around the neighborhood of 
$q=\exp(\pi\I \frac{4\tau-1}{4\tau})\to 0$ 
as in (\ref{OurQ}) hints us how to convert necessary identities 
into the algebraic situation.
As shown in (\ref{a(q)}), the parameter $a$ should be given by 
$$
a^2  =\frac{2\al(q)^2\bal(q)^2}{\al(q)^4+\bal(q)^4}\in K
$$
through which we define the curve $E_a$ over $K$.
We next convert the analytic uniformization of $E_a$ 
(Proposition \ref{EdwardsThm}) by the complex $z$-plane
to the algebraic uniformization by $\bK^\times$ in Tate form.
We first convert the Jacobian theta functions (\ref{theta-def}) 
into the reduced algebraic series $\btheta_1(t,-q^2),
\dots,\btheta_4(t,-q^2)$ by substituting 
$q_\tau\mapsto -q^2$, $q_z{\,\color{black} \mapsto\,} t$ and by dropping the 
fragment factors $-\I (-q^2)^{1/4}$, $(-q^2)^{1/4}$ 
from $\theta_1,\theta_2$ respectively:
\begin{equation}
\label{algtheta-def}
\begin{dcases}
\btheta_1(t,-q^2)&= 
t\sum_{n\in\Z} (-1)^n q^{2n^2+2n}t^{2n}, \\
\btheta_2(t,-q^2)&= 
t \sum_{n\in\Z}
q^{2n^2+2n}t^{2n}, \\
\btheta_3(t,-q^2)&= \sum_{n\in\Z}
(-1)^n q^{2n^2} t^{2n}, \\
\btheta_4(t,-q^2)&=  \sum_{n\in\Z}
q^{2 n^2}t^{2n} .
\end{dcases}
\end{equation}
It turns out from Fricke's transformation (\ref{FrickeTF}) that
those fragment factors from $\theta_1,\theta_2$ cancel each
other and that the algebraic uniformization is
given by 
\begin{equation}
\label{x_t-y_t}
x(t)=- \frac{\btheta_1(t,-q^2)}{\btheta_2(t,-q^2)} , \quad
y(t)=\frac{\btheta_3(t,-q^2)}{\btheta_4(t,-q^2)} 
\qquad (t\in \bK^\times).
\end{equation}
Noting the identities 
\begin{equation}
\label{theta_xy_tq}
\begin{cases}
\btheta_1(t,-q^2)&=t\btheta_3(tq,-q^2),\\
\btheta_2(t,-q^2)&=t\btheta_4(tq,-q^2),\\
\btheta_3(t,-q^2)&=-tq\btheta_1(tq,-q^2),\\
\btheta_4(t,-q^2)&=tq\btheta_2(tq,-q^2);
\end{cases}
\text{ and }
\begin{cases}
x(t)&=x(-t), \\
y(t)&=y(-t),\\
x(tq)&=y(t),\\
y(tq)&=-x(t).
\end{cases}
\end{equation}
we see that the Tate uniformization map $\bK^\times \to E_a(\bK)$
by $t\mapsto (x(t),y(t))$ gives rise to 
$$
E_a(\bK)=\bK^\times/\la \pm q^{4\Z} \ra .
$$
We then apply the fractional substitutions in the same way
as (\ref{substitute}) 
\begin{equation} 
\label{eq4.4}
x=\frac{r\x+\bal(q)}{s\x+\al(q)}, \quad
y=\frac{r\y+\bal(q)}{s\y+\al(q)}.
\end{equation}
for given constants $r,s \in\bK$ with
$\bal(q) s \ne \al(q) r$ to obtain
the curve (\ref{special-equation})
\begin{equation}
f_{r,s}(\x,\y)
=d_{12}(\x+\y)+d_{34}(\x^2+\y^2)+d_5 \x \y
+d_{67}(\x^2 \y+\y^2\x)+d_8 \x^2\y^2 =0
\end{equation}
with coefficients $d_{12},d_{34},d_5,d_{67},d_8$ given as in
Introduction (\ref{d_ij}).
To obtain the tropical curve $C(\mathrm{trop}(f_{r,s}))$ is thus reduced
to evaluating the points $(\x,\y)$ in $(\bK^\times)^2$
by the valuation map
$\mathrm{val}:(\bK^\times)^2\to \Q^2$ that applies 
$v_\bK:\bK\to \Q\cup \{\infty\}$
to each component $\x,\y$ of (\ref{eq4.4}) through: 
\begin{equation}
\label{transofrmed_xy}
\x=\frac{\al(q) x-\bal(q)}{-s x +r},\quad
\y=\frac{\al(q) y -\bal(q)}{-s y+r}. 
\end{equation}
In order to prove Theorem \ref{thm1.2}, 
it suffices to get partial information
on (the
cycle part of) $C(\mathrm{trop}(f_{r,s}))$ 
by performing a procedure parallel to 
(\ref{xy-eval}).
We first observe from (\ref{x_t-y_t}) and (\ref{transofrmed_xy}) that
\begin{align}
\label{yyt}
\y(t)
&= 
\frac{(\al-\bal)(\btheta_4+\btheta_3)-(\al+\bal)(\btheta_4-\btheta_3)}
{(r-s)(\btheta_4+\btheta_3)+(r+s)(\btheta_4-\btheta_3)} \\
&=
\frac{(\al-\bal)(\sum_{n:even}t^{2n}q^{2n^2})
-(\al+\bal)(\sum_{n:odd}t^{2n}q^{2n^2})}
{(r-s)(\sum_{n:even}t^{2n}q^{2n^2})+(r+s)(\sum_{n:odd}t^{2n}q^{2n^2})} .
\nonumber
\end{align}
Recall that the valuation $v_\bK$ of $\bK$ is normalized as
$v_\bK(q)=1$. 
To introduce a variable $u$ for the valuation $v_\bK(t)$, let us take into
account the ultradiscrete limits discussed in \S \ref{sect.3}
where the variable $u$ is normalized as 
$$
t=q_z=\exp({\pi\I \frac{-u-v\I}{2\tau}})
=\exp\left(-2\pi\left(\frac{\theta}{\eps}+\frac{\eps}{\theta}\right)
{\color{black} u}\right),
\quad
\lim_{\eps\to 0} \eps \log(t)=- 2\pi \theta u,
$$
whereas
$$
q =\exp\left(\pi \I \frac{4\tau -1}{{\color{black} 4}\tau}\right)
=\exp(-\frac{\pi \theta}{\eps}), \quad \lim_{\eps\to 0} \eps \log(q)=-\pi \theta.
$$
Thus, let us set
\begin{equation}
\label{u-scale}
v_\bK(t)
\bigl(=2v_\bK(q) u\bigr)
=2u 
\qquad (u\in \Q).
\end{equation}
We shall first 
evaluate $v_\bK(\y(t))$ when the numerator and denominator
have a unique term with distinguished valuation respectively.
First note that $v_\bK(\al-\bal)=1$, $v_\bK(\al+\bal)=0$.
Since 
$v_\bK(t^{2n}q^{2n^2})=2(n+u)^2-2u^2$, we find that  
$$
v_\bK(\sum_{n:even}t^{2n}q^{2n^2})\ge
2(2\lfloor \frac{u+1}{2}\rfloor-u)^2-2u^2
=-\ThetaEven(u)-2u^2
$$
with equality held for $u\in \Q\setminus (1+2\Z)$
and that
$$
v_\bK(\sum_{n:odd}t^{2n}q^{2n^2})\ge
2(2\lfloor \frac{u}{2}\rfloor+1-u)^2-2u^2
=-\ThetaOdd(u)-2u^2
$$
with equality held for $u\in \Q\setminus 2\Z$.
Therefore,
\begin{align} \label{vyt}
v_\bK(\y(t))&=\mathrm{min}\Bigl(1-\ThetaEven(u), -\ThetaOdd(u)\Bigr) \\
&\quad
-\mathrm{min}\Bigl(v_\bK(r-s)-\ThetaEven(u), v_\bK(r+s)-\ThetaOdd(u)\Bigr)
\nonumber
\end{align}
for $u=\frac12 v_\bK(t)\in\Q$ not belonging to the exceptional
set 
\begin{equation}
\Xi_{r,s}:=\Z\cup
\left\{u\in \Q\left|
\begin{matrix*}[c] &1-\ThetaEven(u)= -\ThetaOdd(u),\  \\ 
&v_\bK(r-s)-\ThetaEven(u) = v_\bK(r+s)-\ThetaOdd(u)
\end{matrix*}
\right.\right\}.
\end{equation}
As for $v_\bK(\x(t))$, noting that 
(\ref{theta_xy_tq}) and (\ref{transofrmed_xy}) imply
\begin{equation}
\x(t)=\y(tq^{-1})
\label{xfromy}
\end{equation}
and that $v_\bK(tq^{-1})=2u-1=2(u-\frac12)$, 
we easily derive from (\ref{vyt}) the formula:
\begin{align} \label{vxt}
v_\bK(\x(t))&=\mathrm{min}\Bigl(1-\ThetaEven(u-\frac12), 
-\ThetaOdd(u-\frac12)\Bigr) \\
&\quad
-\mathrm{min}\Bigl(v_\bK(r-s)-\ThetaEven(u-\frac12), v_\bK(r+s)
-\ThetaOdd(u-\frac12)\Bigr)
\nonumber
\end{align}
for $u=\frac12 v_\bK(t)$ such that $u-\frac12\not\in\Xi_{r,s}$.
Thus, 
$C(\mathrm{trop}(f_{r,s}))\subset \R^2$ contains 
the particular set 
\begin{equation}
\left\{\left.
\begin{matrix*}[l]
\mathrm{val}(\x(t),\y(t))\in\Q^2 \\
\bigl(
=(v_\bK(\x(t)),v_\bK(\y(t)))\bigr)
\end{matrix*}
\right|
\tfrac12 v_\bK(t)\not\in \Xi_{r,s}\cup \bigl(\tfrac12+\Xi_{r,s}\bigr),
\quad t\in \bK^\times
\right\}
\end{equation}
whose coordinates are explicitly known by 
(\ref{vxt}) and (\ref{vyt}).
It is then not difficult to see that the closure of the above set 
in the Euclidean plane $\R^2$ is 
equal to the locus of the points $(-X(u),-Y(u))$ with $u\in \R$
given in Theorem 1.2 (iii), 
after noting the general 
equality $-\mathrm{max}(A,B)=\mathrm{min}(-A,-B)$.
The proof of Theorem 1.2 is completed.
\qed
\begin{proof}[Proof of Corollary \ref{cor1.3}]
Since the cycle part of $C(\mathrm{trop}(f_{r,s}))$ is parametrized 
as $(-Y(u-\frac12),-Y(u))$ for $u\in \R$
by Theorem \ref{thm1.2}, each side of the cycle 
can be captured by the shape of the piecewise linear 
graph 
$$
Z=Y(u)= \mathrm{max}\left(\ThetaOdd(u),-1+\ThetaEven(u)\right) 
-\mathrm{max}\left(\delta +\ThetaEven(u), \ThetaOdd(u)
\right)
-v_\bK(r+s),
$$
and its translation $Z=Y(u-\frac12)$
{\color{black} on (u,Z)-plane}. 
Setting
$$
Y_\delta(u):=\mathrm{max}\left(\ThetaOdd(u),-1+\ThetaEven(u)\right) 
-\mathrm{max}\left(\delta +\ThetaEven(u), \ThetaOdd(u)
\right),
$$
so that 
$$
\binom{-Y(u-\frac12)}{-Y(u)}
=\binom{-Y_\delta(u-\frac12)}{-Y_\delta(u))} +
v_\bK(r+s)\binom{1}{1},
$$
we find that the shape of the cycle part of $C(\mathrm{trop}(f_{r,s}))$
depends only on $\delta$.
The rest is not difficult after tracing the move of 
the cycle
$(-Y_\delta(u-\frac12),-Y_\delta(u))_{u\in\R}$ 
on $XY$-plane under
the parameter $\delta$ varied in $\R$ 
(summarized in the following figures for $\delta=-2.5, -1.5, -0.5, 1, 1.5, 2, 2.25$).
\begin{figure}[H]
$\delta=-2.5 \qquad\qquad \delta=-1.5 \qquad\qquad \delta=-0.5$\\
\includegraphics[width=3cm, bb=0 0 487 464]{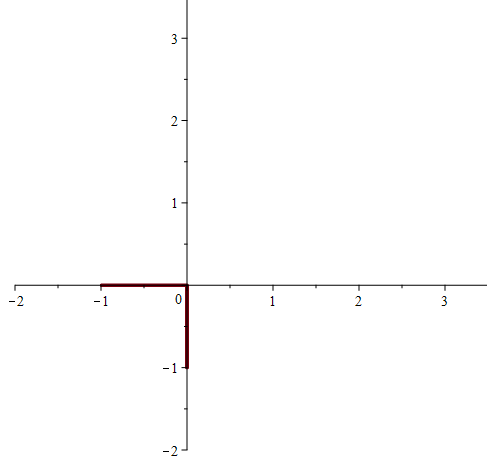}
\includegraphics[width=3cm, bb=0 0 487 464]{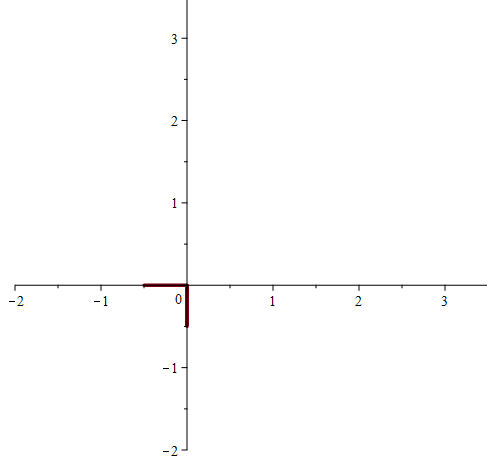}
\includegraphics[width=3cm, bb=0 0 487 464]{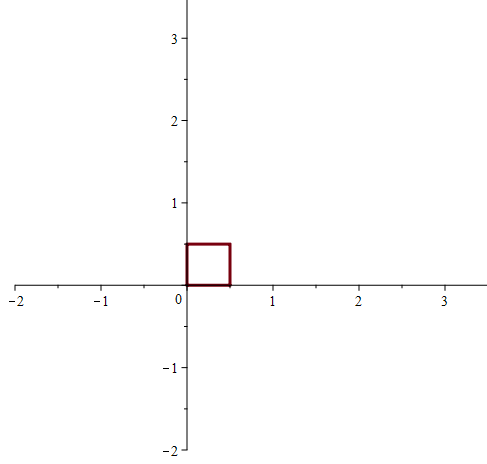} \\
\includegraphics[width=3cm, bb=0 0 487 464]{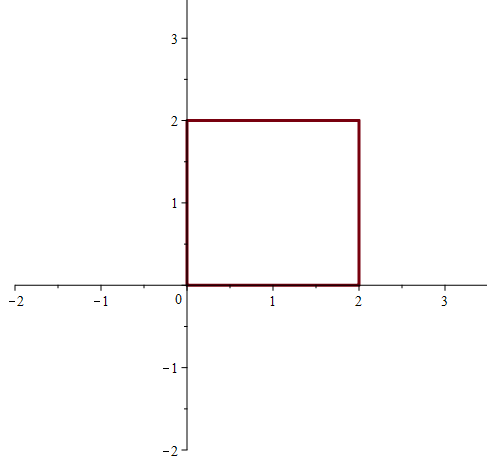} 
\includegraphics[width=3cm, bb=0 0 487 464]{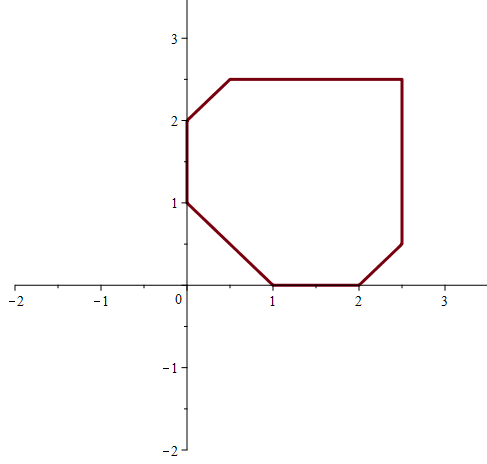} 
\includegraphics[width=3cm, bb=0 0 487 464]{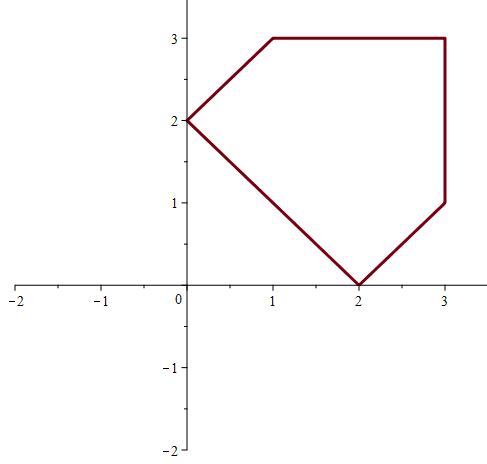} 
\includegraphics[width=3cm, bb=0 0 487 464]{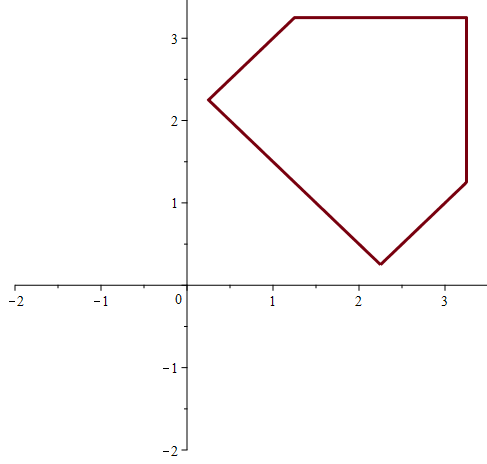} \\
$\delta=1 \qquad\qquad\qquad \delta=1.5 
\qquad\qquad\qquad \delta=2 \qquad\qquad\qquad \delta=2.25$
\end{figure}
\end{proof}
\section{Proof of Proposition 1.4}
In \cite{T}, the second named author studied 
the condition for a truncated symmetric cubic
to have smooth tropicalization.
Let 
$$
f(\x,\y)=
d_{12}(\x+\y)+d_{34}(\x^2+\y^2)+d_5 \x \y
+d_{67}(\x^2\y+\y^2\x)+d_8 \x^2\y^2 
$$
and let $u_{12},u_{34},u_5,u_{67},u_8$ be valuations of 
$d_{12},d_{34},d_5,d_{67},d_8$ respectively.
Then, we have
\begin{Proposition}[\cite{T}]
\label{rani}
Suppose that the tropical elliptic curve $trop(f)$ has 
a polygonal cycle $P$ and let
$(u_{12},u_{34},u_5,u_{67},u_8)$ be the associated parameter
whose entries are 
valuations of the coefficients $d_{12},d_{34},d_5,d_{67},d_8$
respectively. Then, $trop(f)$ is tropically smooth
if and only if $(P; u_{12},u_{34},u_5,u_{67},u_8)$ satisfies
one of the conditions listed in Table \ref{SmoothTropF}:
\begin{table}[hbtp]
  \caption{Cases of smooth $trop(f)$}
  \label{SmoothTropF}
  \centering
  \begin{tabular}{ll}
    \hline
    $P$ & $\quad (u_{12},u_{34},u_5,u_{67},u_8)$    \\
    \hline
    \hline
{{\color{black} \rm [i]} Triangle}   & 
$
\begin{cases*}
-u_{34} + 2 u_{67} - u_8 < 0, \\
u_{12} - u_5 - u_{67} + u_8 < 0, \\
-2u_{12} + 3u_5 - u_8 < 0.
\end{cases*}
$       \\     \hline
{{\color{black} \rm [ii]} Square}  &    $\begin{cases*}
-u_5 + 2u_{67} - u_8 < 0, \\
-u_{12} + 2u_5 - u_{67} < 0, \\
u_{12} - u_{34} - u_5 + u_{67} < 0.
\end{cases*}$ \\      \hline
{{\color{black} \rm [iii]} Pentagon} &   $\begin{cases*}
u_5 - 2u_{67} + u_8 < 0, \\
-u_{12} + u_5 + u_{67} - u_8 < 0, \\
u_{12}- u_{34} - u_5 + u_{67} < 0.
\end{cases*}$      \\      \hline
{{\color{black} \rm [iv] }Hexagon}   &  $\begin{cases*}
-u_5 + 2u_{67} - u_8 < 0, \\
-u_{34} + u_5 < 0, \\
-u_{12} + u_{34} + u_5 - u_{67} < 0.
\end{cases*}$  \\      \hline
{ {\color{black} \rm [v] }Heptagon} &   $\begin{cases*}
u_5 - 2u_{67} + u_8 < 0, \\
-u_{34} + 2u_{67} - u_8 < 0, \\
-u_{12} + u_{34} + u_5  - u_{67} < 0.
\end{cases*}$  \\ 
      \hline
       \end{tabular}
\end{table}
\end{Proposition}
\begin{proof}
This follows from a close look at the ``subdivision''
associated to $\mathrm{trop}(f)$.
For more details, we refer the reader to 
\cite[5.4]{T}.
\end{proof}
\begin{proof}[Proof of Proposition \ref{prop1.4}]
Now, let us prove Proposition \ref{prop1.4} by combining
Corollary \ref{cor1.3} and the above
Proposition \ref{rani}.
Note that, for our two-parameter family $f_{r,s}$,
only cycles of squares, heptagons or pentagons can occur
as observed in Corollary \ref{cor1.3} .
We discuss case by case according to the 
value of $\delta=\delta_{r,s}$.
(i) Suppose that $P$ is a pentagon, i.e., $2\le \delta$.
Then, the possible subdivisions producing a pentagon
are the following two types:
\begin{center}
\begin{minipage}{5cm}
\begin{figure}[H]
\centering
\begin{tikzpicture}
\node [roundnode] (a) at (1,0) {};
\node [roundnode] (b) at (2,0) {};
\node [roundnode] (c) at (2,1) {};
\node [roundnode] (d) at (2,2) {};
\node [roundnode] (e) at (1,2) {};
\node [roundnode] (f) at (0,2) {};
\node [roundnode] (g) at (0,1) {};
\node [roundnode] (h) at (1,1) {};
\draw (a) -- (b) -- (c) -- (d) -- (e) -- (f) -- (g) -- (a);
\draw (h) -- (a);
\draw (h) -- (b);
\draw (h) -- (d);
\draw (h) -- (f);
\draw (h) -- (g);
\end{tikzpicture}
$\quad,$
\end{figure}
\end{minipage}
\begin{minipage}{5cm}
\begin{figure}[H]
\centering
\begin{tikzpicture}
\node [roundnode] (a) at (1,0) {};
\node [roundnode] (b) at (2,0) {};
\node [roundnode] (c) at (2,1) {};
\node [roundnode] (d) at (2,2) {};
\node [roundnode] (e) at (1,2) {};
\node [roundnode] (f) at (0,2) {};
\node [roundnode] (g) at (0,1) {};
\node [roundnode] (h) at (1,1) {};
\draw (a) -- (b) -- (c) -- (d) -- (e) -- (f) -- (g) -- (a);
\draw (h) -- (g);
\draw (h) -- (e);
\draw (h) -- (d);
\draw (h) -- (c);
\draw (h) -- (a);
\draw[thick,dashed]  (g) -- (e);
\draw[thick,dashed]  (a) -- (c);
\end{tikzpicture}
$\quad,$
\end{figure}
\end{minipage}
\end{center}
\noindent
where dashed lines may or may not exist.
Since a smooth tropical curve corresponds to 
a subdivision by triangles of area $\frac12$, 
the smooth case occurs only from the latter type.
Both of these subdivisions produce pentagons as dual graphs
consisting of three right angles and two obtuse angles,
however their shapes differ from each other in that
the former has separated obtuse angles (like the baseball homeplate) 
while the latter has adjacent obtuse angles. 
On the other hand, the explicit parametrization
of the cycle given in Theorem \ref{thm1.2} claims
that only the former type of pentagon occurs in our family, 
all of which turn our to correspond to 
non-smooth tropical curves.
(ii) Suppose next that $P$ is a heptagon, i.e., 
$1< \delta <2$.
In this case, the corresponding subdivision 
is pictured as follows.
\begin{figure}[H]
\centering
\begin{tikzpicture}
\node [roundnode] (a) at (1,0) {};
\node [roundnode] (b) at (2,0) {};
\node [roundnode] (c) at (2,1) {};
\node [roundnode] (d) at (2,2) {};
\node [roundnode] (e) at (1,2) {};
\node [roundnode] (f) at (0,2) {};
\node [roundnode] (g) at (0,1) {};
\node [roundnode] (h) at (1,1) {};
\draw (a) -- (b) -- (c) -- (d) -- (e) -- (f) -- (g) -- (a);
\draw (h) -- (g);
\draw (h) -- (f);
\draw (h) -- (e);
\draw (h) -- (d);
\draw (h) -- (c);
\draw (h) -- (b);
\draw (h) -- (a);
\end{tikzpicture}
\end{figure}
\noindent
This produces a smooth tropical curve.
(iii) Suppose that $P$ is a square, i.e., $-1<\delta \le 1$.
By the result of \cite{V09}-\cite{KMM} and Proposition \ref{thm1.1}, the smooth tropical curve has a cycle of length 8, hence
by Corollary \ref{cor1.3}, it can occur only when $\delta=1$.
Then, after dividing $r,s$ by $q^{v_\bK(r+s)}$, without loss of
generality, we may assume that the parameters $r,s$ are 
of the form:
\begin{align*}
r&=1+r_0q +\sum_{k\ge 1} r_k q^{1+\frac{k}{N}},  \\
s&=-1+s_0q +\sum_{k\ge 1}s_kq^{1+\frac{k}{N}}
\end{align*} 
with $r_0+s_0\ne 0$ for some integer $N>0$.
By simple computations, we find
$u_{12}=1$, $u_5= 0$ and $u_{67}=1$ hold
independently of the choice of $N$.
Then, the smoothness condition (for the square cycle case) 
in Proposition 
{\color{black} 5.1 (the third inequality of Table 1 [ii])}
implies
$$
2<u_{34}=1+v_\bK(\bal^2s^2-\al^2r^2)=
v_\bK{\color{black} \Bigl(
(-4-2s_0-2r_0)q^2+O(q^{2+\frac{1}{N}}
)  \Bigr)},
$$
hence $r_0+s_0=-2$.
Conversely, if $r_0+s_0=-2$, then
$$
v_\bK(\al^2s^2-\bal^2r^2)=
v_\bK{\color{black} \Bigl(
(4-2s_0-2r_0)q}+O(q^{1+\frac{1}{N}}){\color{black} \Bigr)}=1 
,
$$
so that $u_8=u_{34}>2$. This satisfies the smoothness 
condition given in Proposition {\color{black} 5.1 [ii]}.
\end{proof}
\section{Examples}
By virtue of Speyer's work 
\cite[Sect. 7]{S14}, 
the structure of $C(\mathrm{trop}(f_{r,s}))$ 
as a metric graph can be viewed as the projection from a certain
subtree of the Bruhat-Tits tree, if the divisors
(zeros and poles) of the elliptic functions $\x$, $\y$
are known on 
the curve
\begin{equation}
\label{EdwardsEq61}
E_{r,s}: f_{r,s}(\x,\y)
=d_{12}(\x+\y)+d_{34}(\x^2+\y^2)+d_5 \x \y
+d_{67}(\x^2 \y+\y^2\x)+d_8 \x^2\y^2 =0
\end{equation}
with coefficients $d_{12},d_{34},d_5,d_{67},d_8$ given as in (\ref{d_ij}).
We investigate those divisors in view of the theta uniformization 
${\color{black} \wp \colon}
\bK^\times\twoheadrightarrow 
\bK^\times/\la\pm q^{4\Z}\ra
=E_{r,s}(\bK)$ 
determined by the pair $(\x(t),\y(t))$ of functions in $t\in \bK^\times$
(given by (\ref{yyt}) and (\ref{xfromy})) as 
in \S \ref{algebraicProof}, and illustrate 
$C(\mathrm{trop}(f_{r,s}))$ in some special cases of 
parameters $r,s\in \bK$.
In the following examples, we content ourselves with observing smooth cases,
while we hope to look into details of  (subtle) phenomena 
appearing in non-smooth cases in a future separate article.
\subsection{Bruhat-Tits tree}
Let $E$ be the elliptic curve over $\bK$, the smooth completion of the
affine curve defined by $f_{r,s}(\x,\y)=0$.
Noting that $\bK$ is assumed to be an algebraically closed field
of characteristic $0$, we may compose the above 
theta parametrization with the square power map to fit
in the usual form
of Tate uniformization
\begin{equation}
\label{TateUni}
\bK^\times 
{\color{black} \overset{\wp}{ {\relbar\joinrel\twoheadrightarrow} }}
\,
\bK^\times/\la \pm q^4\ra 
{\color{black} =E_{r,s}(\bK) 
\underset{t\mapsto t^2}{\overset{\sim}{\longrightarrow}}
\bK^\times/\la q^8\ra .}
\end{equation}
Denote by $Z_\x$ (resp. $Z_\y$) the set of zeros in $t\in \bK^\times$
of the function $\x\circ \wp$ (resp. $\y \circ \wp)$)
and by $P_\x$ (resp. $P_\y$) the set of poles of $\x\circ\wp$ 
(resp. $\y\circ\wp$), and set 
$$
Z:=Z_\x\cup Z_\y, \quad P:=P_\x \cup P_\y.
$$
Note that from (\ref{xfromy}) we have
\begin{equation}
\label{zeros-poles-xy}
Z_\x=qZ_\y, \quad P_\x=qP_\y. 
\end{equation}
We also see from the equation (\ref{EdwardsEq61}) that both
$\x$ and $\y$ are rational functions of degree $2$ 
on $E_{r,s}$, hence that each of the sets
$Z_\x$, $P_\x$, $Z_\y$, $P_\y$ 
has the image of cardinality at most $2$
under the projection (\ref{TateUni}).

Basic tools devised in \cite{S14} are the Bruhat-Tits $\Q$-tree $BT(\bK)$
and its completion $\overline{BT}(\bK)$
with ends in $ \mathbb{P}^1(\bK)$:
$$
\overline{BT}(\bK)=BT(\bK)\cup \mathbb{P}^1(\bK).
$$
The group $\GL_2(\bK)$ acts naturally on $\overline{BT}(\bK)$
so that the multiplication by $\xi\in \bK^\times$ on 
$\mathbb{P}^1(\bK)=\bK\cup\{\infty\}$ 
is extended to the action of $\bmatx{\xi}{0}{0}{1}\in\GL_2(\bK)$.
We consider $Z\cup P$ as an infinite subset of $\mathbb{P}^1(\bK)$
(necessarily stable under the action of $\la \pm q^4\ra\subset \bK^\times$) 
and its spanning tree
\begin{equation}
\Gamma_{r,s}:=\bigcup_{z,z'\in Z\cup P} [z,z'] \quad (\subset BT(\bK))
\end{equation}
which has, for every point $z\in Z\cup P$ 
$(\subset \bK^\times\subset  \mathbb{P}^1(\bK))$,
a semi-infinite path to the `end' $z$.
It turns out that $\Gamma_{r,s}$ contains an infinite central road $[0,\infty]$.
The metric structure on the internal edges of $\Gamma_{r,s}$ is determined 
by the rule that, for every 4 points $w,x,y,z\in P\cup Z$, 
the length of the internal edge $[w,x]\cap [y,z]$
is given by $|v_\bK(c(w,x:y,z))|$, the valuation of the cross ratio $c(w,x:y,z)$ defined by
\begin{equation}
\label{crossratio}
c(w,x:y,z)=\frac{(w-y)(x-z)}{(w-z)(x-y)}
\end{equation}
(cf.\,\cite[Lemma 4.2]{S14}).
Now, in regards of the above Tate uniformization 
(\ref{TateUni}), $-1\in \bK^\times$ acts on $\Gamma_{r,s}$
by switching two sides 
of the central line so that the quotient tree 
$\overline{\Gamma}_{r,s}:=\Gamma_{r,s}/\la \pm 1 \ra$
(that corresponds to the projection image of $\Gamma_{r,s}$ by 
$\bK^\times\to \bK^\times$ by
$t\mapsto t^2$) projects onto the dense image in $C(\mathrm{trop}(f_{r,s}))$.
\begin{Lemma}
Suppose that $C(\mathrm{trop}(f_{r,s}))$ is tropically smooth. 
Let $\R\overline{\Gamma}_{r,s}$ is the $\R$-tree naturally extended from 
the $\Q$-tree $\overline{\Gamma}_{r,s}$.
Then the tropical curve $C(\mathrm{trop}(f_{r,s}))$ is isometric to  
$\R\overline{\Gamma}_{r,s}/\la q^8 \ra$.
\end{Lemma}
\begin{proof}
This follows easily from the argument in \cite[Sect. 7]{S14}.
The smoothness assumption is used to apply \cite[Theorem 5.7]{J20}
to see the fully faithfulness of the tropicalization map.
\end{proof}
\subsection{The set $Z$ of zeros of $\x\circ\wp,\y\circ\wp$}
The zeros of the function
\begin{equation}
\label{yy(t)theta}
\y(t)=\frac{\al y -\bal}{-s y+r}
=\frac{\al(q)\bar\theta_3(t,-q^2)-\bal(q)\bar\theta_4(t,-q^2)}{-s(q)\bar\theta_3(t,-q^2)+r(q)\bar\theta_4(t,-q^2)}
\end{equation}
(independent of the choice of $r,s$)
are obtained from the equation 
$\bal/\al=\bar\theta_3(t,-q^2)/\bar\theta_4(t,-q^2)$ in $t$,
which is equivalent to
$$
\prod_{n=1}^\infty \frac{1-q^{2n-1}}{1+q^{2n-1}}
=
\prod_{n=1}^\infty 
\frac{(1-q^{4n-2}t^2)(1-q^{4n-2}t^{-2})}{(1+q^{4n-2}t^2)(1+q^{4n-2}t^{-2})}.
$$
It follows that $Z_\y=\{\pm q^{\pm\frac12+4n}\mid n\in\Z\}$
so that $Z_\y$ has the image of cardinality $2$
under the projection (\ref{TateUni}).
By (\ref{zeros-poles-xy}), we conclude
\begin{equation}
Z=Z_\y\cup qZ_\y=
\{\pm q^{{\color{black}\pm \frac{e}{2} +4n}}\mid n\in\Z
{\color{black},\, e=\pm 1,3}\}.
\end{equation}
\subsection{A smooth square case}
We first consider the case
\begin{equation}
\begin{cases}
r(q)&=1-3q, \\
s(q)&=-1+q.
\end{cases}
\end{equation}
To investigate the set of poles of $\y(t)$, set
$$
\Delta:=-s(q)\bar\theta_3(t,-q^2)+r(q)\bar\theta_4(t,-q^2),
$$
and compare it with a product of functions of the theta form
$$
\Theta_{\xi,a}:=\prod_{\substack{n\in\Z\\ 8n+a>0}}(1+\xi q^{8n+a}t^2)
 \prod_{\substack{n\in\Z\\ 8n+a<0}}(1+\xi^{-1} q^{-8n-a}t^{-2})
\quad (a\in {\color{black} \Q}/8\Z,\ \xi\in\bK^\times).
$$
Since $\y(t)$ has degree 2, we may assume 
$
\Delta\sim \Theta_{\xi,a} \cdot \Theta_{\eta,b}
$
for some $a,b\in\Q$, $\xi,\eta\in\bK^\times$
(where $\sim$ means up to $\bK^\times$) so that 
$P_\y=\{\pm\sqrt{-\xi^{-1}}q^{-\frac{a}{2}+4n}, 
\pm\sqrt{-\eta^{-1}}q^{-\frac{b}{2}+4n}\mid n\in\Z 
\}.$
Comparing the logarithmic derivative 
\begin{align*}
\frac{d}{dt}\log(\Delta)=
&\frac{-2 t^{4}+2}{t^{3}} q^{3}
+\frac{-4 t^{4}+4}{t^{3}} q^{4}
+\frac{-8 t^{4}+8}{t^{3}} q^{5}
+\frac{-2 t^{8}-16 t^{6}+16 t^{2}+2}{t^{5}} q^{6} \\
&+\frac{-8 t^{8}-32 t^{6}+32 t^{2}+8}{t^{5}} q^{7}
+\frac{-20 t^{8}-64 t^{6}+64 t^{2}+20}{t^{5}} q^{8}+
\cdots
\end{align*}
with that of $ \Theta_{\xi,a} \cdot \Theta_{\eta,b}$ successively 
from lower degree terms in $q$, 
we find 
\begin{equation*}
\begin{cases}
&a=3,\ \xi=-(1+2 q +3 q^{2}+10 q^{3}+15 q^{4}+38 q^{5}+51 q^{6}+162 q^{7}+\cdots),\\
&b=5,\ \eta= -(1-2 q +q^{2}-6 q^{3}+14 q^{4}-28 q^{5}+84 q^{6}-232 q^{7}+\cdots),
\end{cases}
\end{equation*}
where we only know $\xi,\eta$ as approximate values of $q$-expansions.
By (\ref{zeros-poles-xy}), we conclude
\begin{equation}
P=P_\y\cup qP_\y=
\{\pm\bar\xi q^{\frac52+4n}, 
\pm\bar\eta q^{\frac32+4n},
\pm\bar\xi q^{\frac72+4n}, 
\pm\bar\eta q^{\frac52+4n}
\mid n\in\Z 
\}
\end{equation}
with 
\begin{equation*}
\begin{cases}
\bar\xi:=\sqrt{-\xi^{-1}}=1-q -3 q^{3}+4 q^{4}-10 q^{5}+\frac{55}{2} q^{6}-\frac{153}{2} q^{7}+\cdots, \\
\bar\eta:=\sqrt{-\eta^{-1}}=1+q +q^{2}+4 q^{3}+3 q^{4}+12 q^{5}+\frac{5}{2} q^{6}+\frac{109}{2} q^{7}+\cdots.
\end{cases}
\end{equation*}
We compute various cross ratios (\ref{crossratio}) and their 
valuations in $v_\bK$. For example, we have:
the length of  $[0,\bar\xi q^{-\frac12}]\cap [\bar\eta q^{\frac52}, \infty]=
v_\bK(c(0,\bar\xi q^{-\frac12}
{\color{black} \colon} \bar\eta q^{\frac52},\infty))=3$;
the length of  $[\bar\xi q^{-\frac12},\bar\xi q^{\frac52}]\cap [\bar\eta q^{\frac52}, \infty]=
v_\bK(c(\bar\xi q^{-\frac12},\bar\xi q^{\frac52}
{\color{black} \colon}\bar\eta q^{\frac52}, \infty))=4$ etc.
\noindent
Eventually, we find the shape of $\Gamma_{r,s}$ 
as in the following picture, where each of the vertical (resp. horizontal) internal edges between
two adjacent $\bullet$'s has length one (resp. $\frac12$). 
\begin{align}
&\Gamma_{r,s}: \\
&\xygraph{
{0 \cdots\ }    ([]!{+(0,-.3)} {}) -^{\leftarrow u} [r]
    \bullet ([]!{+(.3,-.3)} {}) 
      ( - [d] \bullet ([]!{+(.3,0)} {})
(
       - []!{+(-0.5,-0.7)} {} ([]!{+(0,-.3)} {-\bar\xi q^{\frac72}}),
        - []!{+(0.5,-0.7)} {} ([]!{+(0,-.3)} {-q^{\frac72}})
        ),
        - [u] {\bullet} ([]!{+(.3,0)} {})
         (
        - []!{+(0.5,0.7)} {}*+!D{q^{\frac72}} ([]!{+(0,-.3)} {}),
        - []!{+(-0.5,0.7)} {}*+!D{\bar\xi q^{\frac72}} ([]!{+(0,-.3)} {})
)
)
    - [r]
    \bullet ([]!{+(0,-.3)} {}) 
    - [r] 
    \bullet ([]!{+(0,-.3)} {}) (
           - [d] \bullet ([]!{+(.3,0)} {})
(
       - []!{+(-0.5,-0.7)} {} ([]!{+(0,-.3)} {-\bar\xi q^{\frac52}}),
        - []!{+(0.5,-0.7)} {} ([]!{+(0,-.3)} {-\bar\eta q^{\frac52}})
        ),
        - [u] \bullet ([]!{+(.3,0)} {})
         (
        - []!{+(0.5,0.7)} {}*+!D{\bar\eta q^{\frac52}}  ([]!{+(0,-.3)} {}),
        - []!{+(-0.5,0.7)} {}*+!D{\bar\xi q^{\frac52}}  ([]!{+(0,-.3)} {})
)
)            
 - [r]
    \bullet ([]!{+(0,-.3)} {{}^{u=1}}) 
- [r]
    \bullet ([]!{+(0,-.3)} {})  (
        - [d] \bullet ([]!{+(.3,0)} {})
 (
       - []!{+(-0.5,-0.7)} {} ([]!{+(0,-.3)} {-\bar\eta q^{\frac32}}),
        - []!{+(0.5,-0.7)} {} ([]!{+(0,-.3)} {-q^{\frac32}})
        ),
        - [u] \bullet ([]!{+(.3,0)} {})
         (
        - []!{+(0.5,0.7)} {}*+!D{q^{\frac32}}  ([]!{+(0,-.3)} {}),
        - []!{+(-0.5,0.7)} {}*+!D{\bar\eta q^{ \frac32}}  ([]!{+(0,-.3)} {})
)
)            
- [r]
    \bullet ([]!{+(0,-.3)} {})
 - [r]
    \bullet ([]!{+(0,-.3)} {})
      (
                - [d] {}*--D{-q^{\frac12}} ([]!{+(.3,0)}),
                -[u]{}*++D{q^{\frac12}} ([]!{+(.3,0)})
)            
 - [r]
    \bullet ([]!{+(0,-.3)} {{}^{u=0}}) 
- [r]
    \bullet ([]!{+(0,-.3)} {})  (
             - [d] \bullet ([]!{+(.3,0)} {})
(
       - []!{+(-0.5,-0.7)} {} ([]!{+(0,-.3)} {-\bar\xi q^{-\frac12}}),
        - []!{+(0.5,-0.7)} {} ([]!{+(0,-.3)} {-q^{-\frac12}})
        ),
        - [u] \bullet ([]!{+(.3,0)} {})
         (
        - []!{+(0.5,0.7)} {}*+!D{q^{-\frac12}}  ([]!{+(0,-.3)} {}),
        - []!{+(-0.5,0.7)} {}*+!D{\bar\xi q^{-\frac12}}  ([]!{+(0,-.3)} {})
)
)            
       - [r] {\cdots \infty}
}
\notag
\end{align}
\noindent
Noting that $-1\in \bK^\times$ acts by switching the upper and lower sides
of the central line, we see that the quotient of the above tree by $\la \pm 1\ra$
(that corresponds to the projection $\bK^\times\to \bK^\times$ by
$t\mapsto t^2$) forms the following tree:
\begin{align}
&\overline{\Gamma}_{r,s}
\ (\cong 
\Gamma_{r,s}/\la \pm 1\ra): \\
&
\xygraph{
{0 \cdots\ }    ([]!{+(0,-.3)} {}) -^{\leftarrow u} [r]
    \bullet ([]!{+(.3,-.3)} {}) 
     ( 
(   
        ),
        - [u] {\bullet} ([]!{+(.3,0)} {})
         (
        - []!{+(0.5,0.7)} {}*+!D{q^{7}} ([]!{+(0,-.3)} {}),
        - []!{+(-0.5,0.7)} {}*+!D{\bar\xi^2 q^{7}} ([]!{+(0,-.3)} {})
)
)
    - [r]
    \bullet ([]!{+(0,-.3)} {}) 
    - [r] 
    \bullet ([]!{+(0,-.3)} {}) (
        - [u] \bullet ([]!{+(.3,0)} {})
         (
        - []!{+(0.5,0.7)} {}*+!D{\bar\eta^2 q^{5}}  ([]!{+(0,-.3)} {}),
        - []!{+(-0.5,0.7)} {}*+!D{\bar\xi^2 q^{5}}  ([]!{+(0,-.3)} {})
)
)            
 - [r]
    \bullet ([]!{+(0,-.3)} {{}^{u=1}}) 
- [r]
    \bullet ([]!{+(0,-.3)} {})  (
         - [u] \bullet ([]!{+(.3,0)} {})
         (
        - []!{+(0.5,0.7)} {}*+!D{q^{3}}  ([]!{+(0,-.3)} {}),
        - []!{+(-0.5,0.7)} {}*+!D{\bar\eta^2 q^{3}}  ([]!{+(0,-.3)} {})
)
)            
- [r]
    \bullet ([]!{+(0,-.3)} {})
 - [r]
    \bullet ([]!{+(0,-.3)} {})
      (     
        -[u]{}*++D{q} ([]!{+(.3,0)})
)            
 - [r]
    \bullet ([]!{+(0,-.3)} {{}^{u=0}}) 
- [r]
    \bullet ([]!{+(0,-.3)} {})  (
,
        - [u] \bullet ([]!{+(.3,0)} {})
         (
        - []!{+(0.5,0.7)} {}*+!D{q^{-1}}  ([]!{+(0,-.3)} {}),
        - []!{+(-0.5,0.7)} {}*+!D{\bar\xi^2 q^{-1}}  ([]!{+(0,-.3)} {})
)
)            
       - [r] {\cdots \infty}
}
\notag
\end{align}
where the central $u$-line acquires the double metric of the original one while
the other edges keeps the original lengths: Consequently each of the internal edges between
two adjacent $\bullet$'s has length one. 
Thus, after the Tate uniformization
$$
\bK^\times
{\color{black} \overset{\wp}{ {\relbar\joinrel\twoheadrightarrow} }}
\,
\bK^\times/\la \pm q^4\ra 
{\color{black} =E_{r,s}(\bK) 
\underset{t\mapsto t^2}{\overset{\sim}{\longrightarrow}}
\bK^\times/\la q^8\ra,
}
$$ 
we find the tropical curve $C(\mathrm{trop}(f_{r,s}))$ is isometric
to the quotient of the above tree modulo $\la q^8\ra$ 
as shown in the following picture.
\vspace{-2cm}
\begin{figure}[H]
\begin{center}
\includegraphics[width=6cm, bb=0 0 486 466]{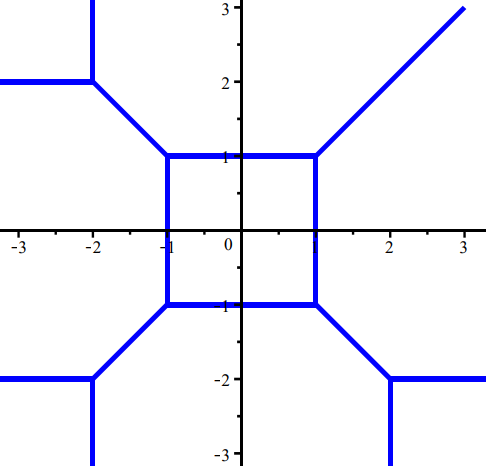}
\end{center}
\end{figure}
\noindent
\subsection{A heptagon case}
We next consider the case
\begin{equation}
\begin{cases}
r(q)&=1+q^{\frac32}, \\
s(q)&=-1+q^{\frac32}.
\end{cases}
\end{equation}
We begin by investigating the set of poles of $\y(t)$ by setting
$$
\Delta:=-s(q)\bar\theta_3(t,-q^2)+r(q)\bar\theta_4(t,-q^2),
$$
and compare it with a product
of theta functions of the form $\Theta_{\xi,a}\Theta_{\eta, b}$.
By analogous consideration to the above square case, 
in this heptagon case, we are led to finding 
$a=\frac72, b=\frac92$ and
comparison of the form
$
\Delta\sim \Theta_{\xi,\frac72} \cdot \Theta_{\eta,\frac92}
$
for some $\xi,\eta\in\bK^\times$
so that 
$$
P_\y=\{\pm\sqrt{-\xi^{-1}}q^{-\frac{7}{4}+4n}, 
\pm\sqrt{-\eta^{-1}}q^{-\frac{9}{\color{black} 4}+4n}\mid n\in\Z 
\}.
$$
Comparing the logarithmic derivative 
\begin{align*}
\frac{d}{dt}\log(\Delta)=
&\frac{2 t^{4}-2}{t^{3}} q^{\frac72}
+\frac{-2 t^{8}+2}{t^{5}} q^{7}
+\frac{ 4t^{8}-4}{t^{5}} q^{8}
+\frac{2 t^{12}+2 t^{8}- 2t^4-2}{t^{7}} q^{\frac{21}{2}} \\
&+\frac{-6 t^{12}-2 t^{8}+2 t^{4}+6}{t^{7}} q^{\frac{23}{2}}
-\frac{2 (t^4-1)(t^4+1)^3}{t^{9}} q^{14}+
\cdots
\end{align*}
with that of $ \Theta_{\xi,a} \cdot \Theta_{\eta,b}$ successively 
from lower degree terms in $q$, 
we find 
\begin{equation*}
\begin{cases}
&a=\frac72,\ \xi=-(1+ q + q^2 +2 q^{3}+2 q^{4} +5 q^{5}+42 q^{6}+131 q^{7}+\cdots),\\
&b=\frac92,\ \eta= 1+q +2 q^2 + 5q^{3}+14 q^{4}+42 q^{5} +132 q^{6}+ 428 q^{7}+\cdots
.
\end{cases}
\end{equation*}
We may skip computing cross ratios since internal edges 
disappear in $\Gamma_{r,s}$ in this case. Consequently
$\overline{\Gamma}_{r,s}$ has
external rays from the central line
\begin{enumerate}
\item[(i)]
for zeros at $\{q^{-1+8\Z},q^{1+8\Z},q^{3+8\Z}\}$, and
\item[(ii)]
for poles at 
$\{
\bar\eta^2q^{\frac{7}{2}+8\Z}, 
\bar\xi^2q^{\frac92+8\Z},
\bar\eta^2q^{\frac{11}{2}+8\Z}, 
\bar\xi^2q^{\frac{13}{2}+8\Z}
\}$,
\end{enumerate}
where $\bar\xi:=\sqrt{-\xi^{-1}}$, $\bar\eta:=\sqrt{-\eta^{-1}}$.
It follows then that $\R\overline{\Gamma}_{r,s}$ looks like the following
tree:
\vspace{-0.5cm}
\begin{figure}[H]
\begin{center}
\begin{tabular}{cc}
\includegraphics[width=12cm]{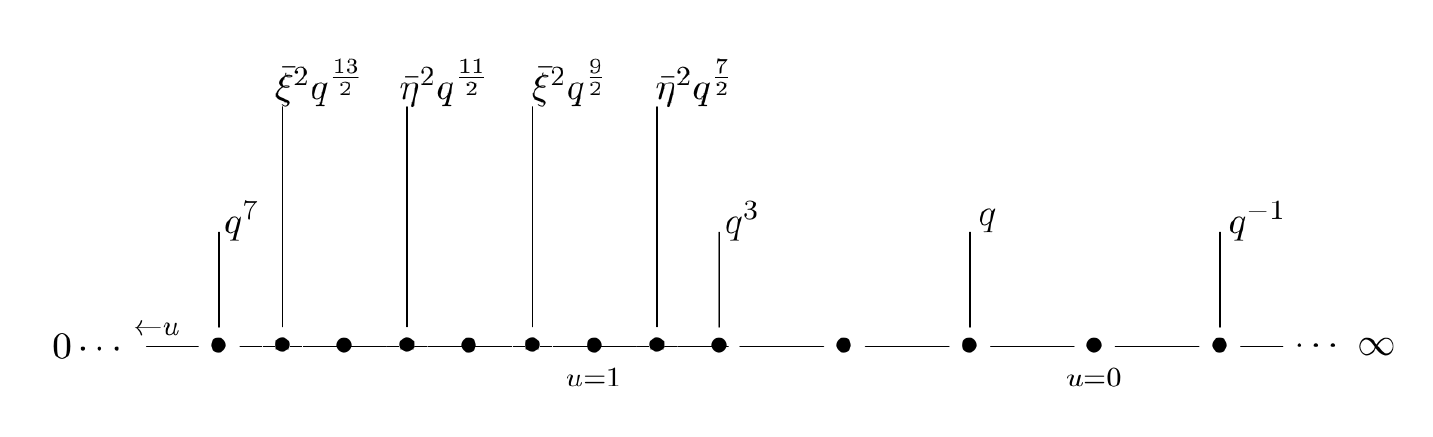} 
\end{tabular}
\end{center}
\end{figure}
\vspace{-1cm}
\noindent
whose quotient modulo $\la q^8\ra$
projects onto the
tropical curve $C(\mathrm{trop}(f_{r,s}))$.
We observe that it is indeed isometric to the following tropical 
curve.
\vspace{-2cm}
\begin{figure}[H]
\begin{center}
\includegraphics[width=6cm, bb=0 0 584 578]{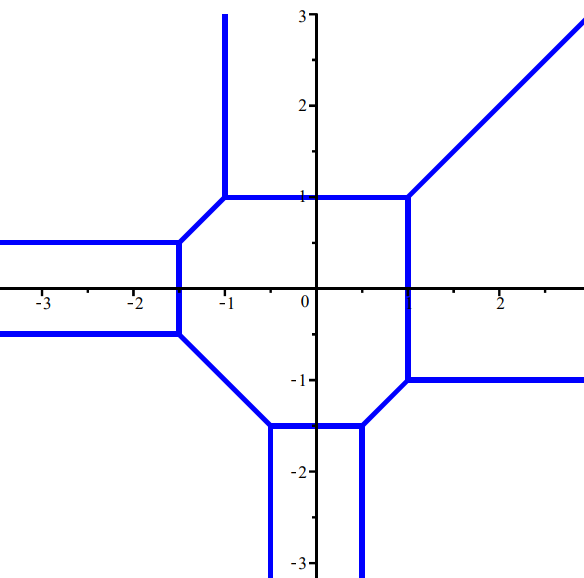}
\end{center}
\end{figure}

\ifx\undefined\bysame
\newcommand{\bysame}{\leavevmode\hbox to3em{\hrulefill}\,}
\fi

\end{document}